\documentclass[]{article}
\usepackage{amssymb}
\usepackage[inline]{enumitem}
\usepackage{hyperref}
\usepackage{url}
\usepackage{amsmath}
\usepackage{algorithm}
\usepackage{algpseudocode}
\usepackage{amsthm}
\usepackage{booktabs}
\usepackage{multirow}
\usepackage[l3]{csvsimple}
\usepackage{authblk}
\usepackage{pgfplotstable}
\pgfplotsset{compat=1.18}

\usepackage{siunitx}
\newcommand{\numD}[1]{\sisetup{scientific-notation=true,%exponent-product={\cdot},
	output-exponent-marker=e,round-mode=places,round-precision=1,table-format=1.4}\num{#1}}
    \newcommand{\numDi}[1]{\sisetup{scientific-notation=false,%exponent-product={\cdot},
	output-exponent-marker=e,round-mode=places,round-precision=3,table-format=1.4}\num{#1}}

\DeclareMathAlphabet{\mathsfit}{\encodingdefault}{\sfdefault}{m}{sl}
\SetMathAlphabet{\mathsfit}{bold}{\encodingdefault}{\sfdefault}{bx}{n}

\def \sumX{\sum_{x \in X}}
\def \sumY{\sum_{y \in X}}
\def \sumXY{\sum_{(x,y)\in X\times X}}
\def \sumNct{\sum_{(x,y)\in \Nct}}

\def \sumOct{\sum_{y \in \Oct}}
\def \sumIct{\sum_{x \in \Ict}}

\def \forallx{\forall x \in X}
\def \forally{\forall y \in X}
\def \forallxy{\forall (x,y) \in X\times Y}

\def \erre{\mathbb{R}}

\def \Nct{N_t}
 
\def \Oct{B_t(x)}
\def \Ict{B_t(y)}

\def \NFs{\mathcal{N}^{(t)}_{\mu,\nu}}

\def \MNFs{\Theta^{(t)}_{\mu,\nu}}

\def \bside{\mathbb{B}^{(t)}_s}

\def \PP{\mathcal{P}}

\def \MM{\mathcal{M}}

\def \Pimn{\Pi_{\mu,\nu}}

\def \erre{\mathbb{R}}

\def \minpi{\min_{\pi \in \Pimn}}

\def \maxNFs{\max_{\eta\in\NFs}}

\def \mn{(\mu,\nu)}

\def \sumX{\sum_{x \in X}}
\def \sumY{\sum_{y \in X}}
\def \sumXY{\sum_{(x,y)\in X\times X}}
\def \sumNct{\sum_{(x,y)\in \Nct}}

\def \sumOct{\sum_{y \in \Oct}}
\def \sumIct{\sum_{x \in \Ict}}

\def \forallx{\forall x \in X}
\def \forally{\forall y \in X}
\def \forallxy{\forall (x,y) \in X\times X}

\def \Nct{N_t}
 
\def \Oct{B_t(x)}
\def \Ict{B_t(y)}

\def \NFs{\mathcal{N}^{(t)}_{\mu,\nu}}

\def \MNFs{\Theta^{(t)}_{\mu,\nu}}
\def \bside{\mathbb{B}^{(t)}_s}

\DeclareMathOperator*{\argmin}{arg\,min}

\newtheorem{definition}{Definition} 

\newtheorem{corollary}{Corollary} 
 
\newtheorem{remark}{Remark}
\newtheorem{theorem}{Theorem} 

\newcommand{\minimal}{minimal }
\newcommand{\Minimal}{Minimal }

\begin{document}
\date{}
\title{On the computation of the infinity Wasserstein distance and the Wasserstein Projection Problem}
\author[1]{Gennaro Auricchio}%\ead{gennaro.auricchio@unipd.it}
\author[2]{Gabriele Loli}
\author[2]{Marco Veneroni}
\affil[1]{Università di Padova, Dipartimento di Matematica ``Tullio Levi-Civita'', Via Trieste 63, Padua, 35131, Italy}
             
\affil[2]{Università di Pavia, Dipartimento di Matematica ``F. Casorati'', Via A. Ferrata 1, Pavia, 27100, Italy}
\maketitle

%% Abstract
\begin{abstract}
Computing the infinity Wasserstein distance and retrieving projections of a probability measure onto a closed subset of probability measures are critical sub-problems in various applied fields.
However, the practical applicability of these objects is limited by two factors: either the associated quantities are computationally prohibitive or there is a lack of available algorithms capable of calculating them.
In this paper, we propose a novel class of Linear Programming problems and a routine that allows us to compute the infinity Wasserstein distance and to compute a projection of a probability measure over a generic subset of probability measures with respect to any $p$-Wasserstein distance with $p\in[1,\infty]$. 
\end{abstract}

\section{Introduction}

Given a Polish space $X$ endowed with a metric $d:X\times X\to [0,\infty)$, the $p$-Wasserstein distance between two probability measures over $X$, namely $\mu,\nu\in\PP(X)$, is defined as
\begin{equation}
    \label{eq:introWp}
    W_p(\mu,\nu):=\bigg(\min_{\pi\in\Pi(\mu,\nu)}\int_{X\times X} d^p(x,y){\rm d}\pi(x,y)\bigg)^{\frac{1}{p}},
\end{equation}
where $\Pi(\mu,\nu)$ is the set of transportation plans between $\mu$ and $\nu$, that is,
\begin{equation}
\begin{split}
    \label{eq:transportationplans}
	\Pi(\mu,\nu):=\Big\{\pi\in\PP(X\times X):\pi(A\times X)= \mu(A),\\
	\pi(X \times B)=\nu(B),\quad \forall A,B\subseteq X\Big\}.
\end{split}
\end{equation}
Owing to their appealing mathematical properties, the Wasserstein distances and Optimal Transport problems have been used across several applied fields, ranging from Machine Learning  \cite{scagliotti2023normalizing,Cuturi2014,Frogner2015,pmlr-v70-arjovsky17a,Granger2024a} and Computer Vision \cite{Rubner98,Rubner2000,Pele2009,Granger2024b} to Partial Differential Equations \cite{Cipriani_Fornasier_Scagliotti_2025,brenier1991polar,carrillo2007contractive,mather1991action}, minimizing measures in Lagrangian dynamics \cite{evans2002linear,evans2003some,de2006minimal,granieri2007action}, and Algorithmic Game Theory \cite{auricchio2024k,daskalakis2013mechanism,auricchio2024extended}.
The growing popularity of these distances generated an increasing demand for efficient ways to solve the minimization problem that defines them.
Nowadays, there are different approaches to the minimization problem on the right-hand side of \eqref{eq:introWp}: the most popular methods are based on
\begin{enumerate*}[label=(\roman*)]
    \item the Sinkhorn's algorithm (see \cite{Cuturi2013,Solomon2015,Altschuler2017}), which solves a regularized version of the basic optimal transport problem,
    \item Linear Programming-based algorithms (see \cite{LingOkada2007,Auricchio2018,Bassetti2018}), which approximate or exactly solve the basic optimal transport problem by formulating and solving an equivalent uncapacitated minimum cost flow problem \cite{Orlin,Goldberg1989}, and
    \item Linear Programming-based heuristics, which aim at reducing the computational complexity of the LP Algorithms by simplifying the structure of the problem, as done in \cite{Pele2009, Auricchio2018}.
\end{enumerate*}
% 
%All these methods are valuable alternatives when it comes to compute the Wasserstein Distance between a generic couple of probability measures.
% 
% however, not all problems related to Optimal Transports can be handled through these means.
% 

% 
Despite the variety of methods developed to compute Wasserstein distances, there are still classes of subproblems connected to Optimal Transport that are either prohibitive to solve or have remained unaddressed.
Let us consider, for example, the infinity Wasserstein distance $W_{\infty}$, which is obtained by taking the limit of \eqref{eq:introWp} as $p$ goes to infinity.
The infinity Wasserstein distance has been widely studied in \cite{brizzi23,Pascale08,auricchio2022structure,Steinerberger21} and has been fruitfully applied in several fields, such as branching models \cite{Bernot06,maddalena03}, the Vlasov equation \cite{toscani04,carrillo07}, travel policies \cite{butpraste} and also astronomy \cite{McCann06stablerotating}.
However, owing to the fact that the computation of $W_{\infty}$ cannot be tied to a Linear Programming problem, retrieving the infinity Wasserstein distance between two measures is, to the best of our knowledge, still an open problem.
Indeed, to the best of our knowledge, only two other papers tackle the problem of computing the $W_\infty$ distance.
In \cite{bansil2021w}, the authors connect the Wasserstein problem with a perfect matching problem and use this connection to infer a set of conditions under which the optimal transportation plan is induced by a map.
The same paper also proposes a scheme to approximate the value of the infinity Wasserstein using a bisection-like method.
While the first paper has a combinatorial perspective on the problem, the second paper tackles the issue following a Sinkhorn-like approach \cite{Cuturi2013}.
Indeed, the authors use an entropic regularization term to approximate $W_\infty$ and the optimal plan inducing the distance \cite{brizzi2024entropic}.
Similarly, being able to project a probability measure $\mu\in\PP(X)$ over a closed and convex subset of $\PP(X)$ with respect to a Wasserstein distance $W_p$ is a sub-problem in crowd motion models \cite{dimm16,maury10}, where the density of the solution must satisfy an upper bound.
However, computing the exact projection is often impractical, and thus a heuristic approach based on the heat equation is preferred \cite{marino16}.
In this paper, we address this gap in the literature by introducing a class of Linear Programming problems and an iterative routine that, combined together, allow us to compute $W_{\infty}$ and to retrieve a projection of a generic measure $\mu$ onto a subset of probability measures with respect to any $p$-Wasserstein distance $W_p$ with $p\in[1,\infty]$.

\section{Preliminary Notions and Notation}
\label{sect:preliminaries}

In this section, we fix our notation and recall the basic notions of Optimal Transport. 
For a complete introduction to Optimal Transport we refer the reader to \cite{villani2009optimal}.
% 
% about  the Discrete Optimal Transport problem and the $t$-truncated Transportation problem, \cite{Pele2009}.

\subsection{Basic Notions about Discrete Optimal Transport}

Let $X$ be a discrete and finite set of points.
For simplicity, we assume that $X\subset \erre^n$ for a suitable $n\in\mathbb{N}$, and endow $X$ with the Euclidean distance 
\begin{equation}
    |x-y|:=\sqrt{\sum_{i=1}^n(x_i-y_i)^2}.
\end{equation}
Given $f:X\to [0,+\infty)$ and $g:X \times X \to [0,+\infty)$, we set
\[
    |f|_X = \sumX f_x \quad \text{and}\quad |g|_{X\times X}=\sumXY g_{x,y}.
\]
Moreover, we denote with $\MM(X)$ the set of nonnegative real functions on $X$ and with $\PP(X)$ the subset $\{\mu \in \MM(X):|\mu|_X=1\}$; analogously to continuous domains, we refer to these spaces as \emph{positive measures} and \emph{probability measures}, respectively.
Given two probability measures $\mu$ and $\nu$, we define the set of transportation plans between $\mu$ and $\nu$ as
\[
\Pi_{\mu,\nu}:=\bigg\{\pi_{x,y}\geq 0:\sumX \pi_{x,y}=\nu_y;\sumY \pi_{x,y}=\mu_x \bigg\}.
\]
We define the $p$-transportation cost functional $\mathbb{T}_p : \Pi_{\mu,\nu} \rightarrow [0, \infty)$ as $\mathbb{T}_p(\pi):=\sumXY |x-y|^p\pi_{x,y}$. 
The $p$-Wasserstein distance between $\mu$ and $\nu$ is then defined as
\begin{equation}
\label{def:KW}
	W_p^p(\mu,\nu):=\inf_{\pi \in \Pi_{\mu,\nu}}\mathbb{T}_p(\pi).
\end{equation}

Since we are considering a $n$-dimensional finite set of points, the infimum in \eqref{def:KW} is a minimum, which is characterized as the solution to the following Hitchcock-Koopmans transportation problem \cite{Flood1953}:
\begin{align}
\label{eq1:1}    W_p^p(\mu,\nu) := \min \quad & \sumXY |x-y|^p \pi_{x,y} \\
    \mbox{s.t.} \quad 
    & \sumX \pi_{x,y}=\nu_y, & \forally, \\
    & \sumY \pi_{x,y}=\mu_x, & \forallx, \\
\label{eq1:x}    & \pi_{x,y} \geq 0 & \forallxy.
\end{align}

% While this problem is solvable in polynomial time $O(N^3\log(N))$ using the uncapacitated min cost flow algorithms presented in \cite{Orlin,Goldberg1989}, (where $N \sim$ the sum of the number of points in $X$ and the number of points in $Y$) the size of instances appearing in Machine Learning applications makes the solution of this problem a very interesting computational challenge.

\subsection{The Infinity Wasserstein distance}

Given a transportation plan $\pi\in\Pi(\mu,\nu)$, we define $\mathbb{T}_{\infty}(\pi)$ as it follows
\[
\mathbb{T}_{\infty}(\pi):=\max_{x,y\in X\times X, \; \text{s.t.}\; \pi_{x,y}>0}|x-y|,
\]
so that $\mathbb{T}_{\infty}(\pi)$ is the longest mass movement described by the transportation plan $\pi$.
We then define the infinity Wasserstein distance between $\mu$ and $\nu$ as 
\begin{equation}
    \label{eq:Winfinity}
    W_{\infty}(\mu,\nu):=\min_{\pi\in\Pi(\mu,\nu)}\mathbb{T}_{\infty}(\pi),
\end{equation}
hence $W_{\infty}(\mu,\nu)$ describes the \minimal longest transportation that can be accomplished with a transportation plan.
It is easy to check that $W_{\infty}(\mu,\nu)$ is equal to the limit of $W_p(\mu,\nu)$ when $p\to\infty$, \cite{Santambrogio2015}.
Notice that, unlike the Wasserstein distance, the infinity Wasserstein distance between two probability measures cannot be characterized as the minimum value of a minimum cost-flow problem.

In Theorem \ref{thm:alternativeWinfformulation} we will address the problem of computing the $W_\infty$ distance between any couple of probability measures.

\subsection{The Wasserstein Projection Problem}
\label{sec:prelim_projection}
It is well known that $W_p$ induces a metric over $\PP(X)$ for every $p\in[1,\infty]$, \cite{villani2009optimal}. 
\begin{definition}
Given a subset $K\subset \PP(X)$ and an element $\mu \in \PP(X)$, we say that $\nu \in \PP(X)$ is a projection of $\mu$ over $K$, with respect to $W_p$, if
\begin{equation}
    \label{def:projection_nu}
    \nu \in K\quad \text{and}\quad W_p(\mu,\nu) \leq W_p(\mu,\rho)\quad \forall \rho \in K.
\end{equation}
\end{definition}
In what follows, we focus on a specific class of sets $K$.
Given $f\in \MM(X)$ such that $|f|_X\geq 1$, we set
\begin{equation}
\label{eq:K_alpha}
    K_f:=\bigg\{\rho\in\PP(X) : \rho_x\leq f_x,\quad
    \forall x \in X\bigg\}.
\end{equation}
Given any $f\in\MM(X)$, the set $K_f$ is closed, which ensures the existence of at least one projection. With respect to the continuous setting, uniqueness is not to be expected in the discrete one, due to a lack of convexity. We recall that a standard geodesic between two measures $\mu_0$ and $\mu_1$ in the $p$-Wasserstein space on $\erre^n$ is given by
\[
	\mu_t = ((1-t)P^1+tP^2)_\sharp \gamma,\quad t\in [0,1]
\]  
where $P^i:\erre^n \times \erre^n \to \erre^n$ is the projection on the $i^{th}$ component, $\gamma$ is an optimal transport plan between $\mu_0$ and $\mu_1$, and the `$_\sharp$' symbol denotes the push-forward measure. The set $K_f$ is convex (in the classical sense, i.e. $t\mu +(1-t)\nu \in K_f$ for all $\mu,\nu\in K_f$ and $t\in [0,1]$) but, since we are considering the case in which measures are supported over a grid, $K_f$ is not geodesically convex with respect to any $W_p$ distance as long as the grid contains more than two points, since $(1-t)P^1+tP^2 \notin X$ for a generic $t\in(0,1)$ and therefore the geodesic is not defined.
As a consequence, there might be more than one measure satisfying conditions \eqref{def:projection_nu}, as the following example shows. Let 
% is in general false, as the following example shows. Let 
\[
	X=\{0,1,2\}\subset \erre,\quad \mu =\delta_1, \quad f=\frac12(\delta_0+\delta_1+\delta_2).
\]
Then, for
\[
	\nu_t=\frac12(t\delta_0 +\delta_1+(1-t)\delta_2),\quad t\in [0,1]
\]	
 a direct computation shows that 
\[
	W_p(\mu,\nu_t)= \left( \int_X|x-1|^p\ d\nu_t(x)\right)^{\frac1p} = \left( \frac12(t +0+(1-t) )\right)^{\frac1p} =2^{-\frac{1}{p}}
\]
so that all measures $\nu_t$, for $t\in [0,1]$ are minimizers of $W_p(\mu,\cdot)$, for all $p\geq 1$.

In Theorem \ref{thm:projectionclassic} we will address the problem of computing a projection of a measure $\mu$ on $K_f$, with respect to the $W_1$ metric.
Noticeably, the argument used in Theorem \ref{thm:projectionclassic} can be extended to the projection problem induced by any $W_p$ with $p\in(1,\infty)$.
The case $p=\infty$ is handled in Corollary \ref{crll:lastcorollary}.

\section{The Maximum Nearby Flow Problem}

Given a threshold $t \ge 0$, we define the set of nearby points as
\[
\Nct:=\big\{(x,y)\in X\times X : |x-y|\le t \big\},
\]
moreover, given $x\in X$, we denote the radius $t$ ball centred in $x$ as
\[
    \Oct := \big\{ y \in X : |x-y|\le t \big\}.
\]
For every $t\in[0,\infty)$, a utility function over $\Nct$ is a collection of nonnegative values $s^{(t)}:=\{s^{(t)}_{x,y}\}_{(x,y)\in\Nct}$.
Given $\mu\in\PP(X)$, $\nu\in\MM(X)$, $t \ge 0$, and a utility function $s^{(t)}$, the Maximum Nearby Flow (MNF) problem, introduced in \cite{Auricchio2019}, is defined as
%\begin{align}
%    \label{nbf:0} \maxNFs&\bside(\eta):=\maxNFs\sum_{(x,y)\in\Nct}s_{x,y}^{(t)}\eta_{x,y}
% %    \label{nbf:1} &\sumIct \eta_{x,y} \leq \nu_y, \quad\quad\quad \forall y \in Y, \\
%	% \label{nbf:2} &\sumOct \eta_{x,y} \leq \mu_x, \quad\quad\quad \forall x \in X,
%\end{align}
\begin{align}
    \label{nbf:0} \maxNFs&\bside(\eta),
\end{align}
where
\[
	\bside(\eta):=\sum_{(x,y)\in\Nct}s_{x,y}^{(t)}\eta_{x,y}
\]
and $\NFs$ is the set of nearby flows between $\mu$ and $\nu$, that is
\begin{equation}
\begin{split}
    \NFs:=\bigg\{\eta\in\MM(X \times X):\ &\eta_{x,y}=0\;\;\text{if}\;\; (x,y)\notin N_t,\\
    % \text{supp}(\eta)\subseteq N_t,\\ 
    \sum_{x\in X} \eta_{x,y}\leq \nu_y\ &\forall y \in X,  \sum_{y \in X} \eta_{x,y}\leq \mu_x\ \forall x  \in X\bigg\}.
\end{split}    
\end{equation}
Note that, for all $\eta \in \NFs$, for all $\mu\in\PP(X)$ and all $\nu\in\MM(X)$,
\begin{equation}
\label{eq:leq1}
    |\eta|_{X \times X} = \sum_{x\in X}\sum_{y\in X}\eta_{x,y}\leq \min\Big\{\sum_{x\in X} \mu_x,\sum_y\nu_y\Big\} \leq
    % \min\{|\mu|_X,|\nu|_Y\}=
    \min\{1,|\nu|_Y\}.
\end{equation}

For the sake of simplicity and without loss of generality, from now on, we assume that $|\nu|_Y\ge 1$, so that $|\eta|_{X\times Y}\le 1$.
Since $\NFs$ is nonempty (it contains the null nearby flow,  $\eta_{x,y}=0$) and $\Nct$ is a finite set, the MNF problem is feasible and admits a solution.

%\begin{theorem}
%    For every $\mu$, $\nu$, $t$, and $s$, the MNF is feasible and admits a solution.
%\end{theorem}
%
%
%\begin{proof}
%    First, notice that the null nearby flow, defined as $\eta_{x,y}=0$ is always feasible to problem \eqref{nbf:0}.
%    % 
%    In particular, $\NFs$ is non-empty hence the problem either admits a solution or is unbounded.
%    % 
%    To conclude, it then suffices to show that the objective value of the problem is bounded.
%    % 
%    This follows from the fact that $|\eta|_{\Nct}\le \max\{|\mu|_X,|\nu|_X\}$ and $\Nct$ has a finite number of points, thus there exists a positive constant $C_t$ for which it holds $s_{x,y}^{(t)}\le C_t$, therefore
%    \[
%        \max_{\eta\in\NFs}\sum_{(x,y)\in\Nct}s^{(t)}_{x,y}\eta_{x,y}\le C_t\max\{|\mu|_X,|\nu|_X\},
%    \]
%    which allows us to conclude the proof.
%\end{proof}

% 
In what follows, we denote by $\Theta^{(t)}_{\mu,\nu,s}$ the set of optimal solutions to the MNF problem with threshold $t \ge 0$, utility function $s^{(t)}$, probability measure $\mu\in\PP(X)$ and measure $\nu\in\MM(Y)$.
When the utility function is clear from the context, we omit $s$ from the index of $\Theta^{(t)}_{\mu,\nu,s}$.
Even though the Maximum Nearby Flow problem is well defined for any utility function $s^{(t)}$, in what follows we focus on two specific classes of utilities:
\begin{itemize}
    \item the \textit{constant utility} $s^{(t)}_{x,y}=1$ for every $t\ge 0$ and every couple $(x,y)\in\Nct$, which enables us to compute the infinity Wasserstein distance $W_\infty$ and a projection with respect to $W_\infty$. In this case, by \eqref{eq:leq1}
    \begin{equation}
    \label{eq:max1}
        \maxNFs \bside(\eta)= \maxNFs\sum_{(x,y)\in\Nct}s_{x,y}^{(t)}\eta_{x,y} = \maxNFs |\eta|_{N_t}\leq 1,
    \end{equation}
    
    and
    \item the \textit{complementary cost utility}, defined as $s_{x,y}^{(t)}=t-|x-y|$ for every $(x,y)\in\Nct$, which enables us to solve the Wasserstein Projection Problem. Notice that, since $|x-y|\le t$ for every $(x,y)\in\Nct$, then $s^{(t)}_{x,y}\ge 0$.
\end{itemize}

% 
% Since $\mu$ and $\nu$ are two probability measures, we have that $|\eta^{(t)}|_{\Nct}\le 1$ for every $ \eta^{(t)}\in\NFs$ and every $t\ge 0$.
% 
If a nearby flow $\eta\in\NFs$ satisfies $|\eta|_{X \times X} = 1$, 
%we define $\pi\in\PP(X\times X)$ as
%\begin{equation}
%    \label{eq:trivial_extension}
%    \pi_{x,y}=\begin{cases}
%        \eta_{x,y}\quad\quad &\forall (x,y)\in\Nct,\\
%        0 &\text{otherwise,}
%\end{cases}    
%\end{equation}
% 
the inequalities in \eqref{eq:leq1} are equalities and therefore $\eta$ is a transportation plan between $\mu$ and $\nu$, that is, $\eta \in \Pimn$.
We now turn to the problem of finding the smallest value $t$ for which there exists a solution $\bar \eta$ to the MNF such that $|\bar \eta|_{\Nct}=1$.
For this reason, given a utility function $s^{(t)}_{x,y}$, we introduce the \textit{\minimal saturation threshold}.

\begin{definition}[\Minimal Saturation Threshold]
\label{minimalsaturation}
Given a probability measure $\mu\in \PP(X)$, a measure $\nu \in \MM(X)$ such that $|\nu|_X\ge 1$, and a utility function $s^{(t)}_{x,y}$, we define the \emph{\minimal saturation threshold} as
\begin{equation}
\label{minimalsaturationformula}
    \tau_s(\mu,\nu):= \min \bigg\{ t \ge 0\text{ s.t. } \exists \eta \in \MNFs \text{ s.t. } |\eta|_{X \times X} =1 \bigg\},
\end{equation}
where $\MNFs$ is the set containing all the optimal nearby flows between $\mu$ and $\nu$ for threshold $t\ge 0$.
\end{definition}

Notice that, since $s^{(t)}_{x,y}\ge 0$ by definition, if we set $t=\max_{(x,y)\in X\times X}|x-y|$, then there exists at least one optimal nearby flow such that $|\bar \eta|_{X\times Y}=1$.
In particular, the right-hand side of \eqref{minimalsaturationformula} is the infimum of a non-empty set, thus $\tau_s(\mu,\nu)$ is well defined and finite.

\subsection{Characterizing the Infinity Wasserstein distance via the \Minimal Saturation Threshold}

We now use the \Minimal Saturation Threshold to characterize the infinity Wasserstein distance induced by the constant utility function, that is $s_{x,y}^{(t)}=1$ for every $t\ge 0$ and for every $(x,y)\in\Nct$.

\begin{theorem}
\label{thm:alternativeWinfformulation}
Let $\mu$ and $\nu$ be two probability measures on $X$ and let $s^{(t)}$ be the constant utility function, so that $s^{(t)}_{x,y}=1$ for every $t \ge 0$ and every $(x,y)\in\Nct$.
Then,
\[
    W_{\infty}(\mu,\nu)=\tau_s(\mu,\nu).
\]
\end{theorem}

\begin{proof}
Let $t \ge 0$ be such that there exists $\eta \in \MNFs$ for which $|\eta|_{X\times X}=1$. 
Then, $\eta \in\Pimn$, and since supp$(\eta)\subseteq N_t$,  $\mathbb{T}_\infty(\eta)\leq t$.
Therefore, 
\[
    W_{\infty}(\mu,\nu)\leq \min \left\{ t \ge 0\text{ s.t. } \exists \eta \in \MNFs \text{ s.t. } |\eta|_{X \times X} =1 \right\}=\tau_s(\mu,\nu).
\]
In order to conclude we need to prove the inverse inequality.
Let $t=W_{\infty}(\mu,\nu)$,
then there exists $\pi \in \Pi(\mu,\nu)$ such that
\[
    \mathbb{T}_\infty(\pi)=t,\qquad \text{i.e.,}\qquad \pi_{x,y}=0 \qquad\text{if}\qquad |x-y|>t.
\]
%If we define
%\[
%\eta_{x,y}=\pi_{x,y} \qquad\qquad (x,y)\in \Nct
%\]
%we have that
%\[
%|\eta|_{\Nct}=\sum_{(x,y)\in \Nct}\eta_{x,y}=\sum_{(x,y)\in X\times Y}\pi_{x,y}=1.
%\]
Therefore, $\pi\in \NFs$ and $|\pi|_{N_t}=1$. By \eqref{eq:max1}, $\pi\in\MNFs$, and therefore
\[
    \min \Big\{ t \ge 0 \;\text{ s.t. } \; \exists \eta \in \MNFs \; \text{ s.t. } \; |\eta|_{\Nct}=1 \Big\}\leq \mathbb{T}_\infty(\pi)=W_{\infty}(\mu,\nu),
\]
which concludes the proof.
\end{proof}

\subsection{Solving the Projection Problem via the Maximum Nearby Flow problem}

In this section we show that the \minimal saturation threshold associated with the complementary cost utility, defined as $s^{(t)}_{x,y}=t-|x-y|$ for $t\ge 0$ and $(x,y)\in\Nct$, is connected to the computation of the Wasserstein Projection.
Given a threshold $t\ge 0$, we denote with $|\cdot|_t$ the $t$-truncated Euclidean norm, defined as
\[
    |x-y|_t:=\min\{|x-y|,t\}.
\]
We denote with $W_{p}^{(t)}$ the Wasserstein distance obtained by using the $p$-th power of the $t$-truncated Euclidean distance $|\cdot|_t$ rather than the standard Euclidean distance in \eqref{eq1:1}, that is
\[
    W_{p}^{(t)}(\mu,\nu):= \left(\minpi \sumXY |x-y|^p_t\pi_{x,y}\right)^{\frac1p}.
\]
Since $|x-y|_t\leq |x-y|$ for all $x,y\in X$, then
\[
    \sumXY |x-y|^p_t\pi_{x,y} \leq \sumXY |x-y|^p\pi_{x,y}\qquad \forall \pi\in \Pi_{\mu,\nu},\ \forall \mu,\nu\in \PP(X)
\]
and thus
\begin{equation}
\label{eq:w1trunc}
    W_p^{(t)}(\mu,\nu)\leq W_p(\mu,\nu)\qquad \forall \mu,\nu\in \PP(X). 
\end{equation}

\begin{remark}
\label{rmk:equivalence}
    It was shown in \cite{Auricchio2019} that if $\mu$ and $\nu$ are two probability measures and $s_{x,y}^{(t)}=t-|x-y|$ is the complementary cost utility, then 
    \begin{equation}
	\label{eq:thtrunceq}
		W_{1}^{(t)}(\mu,\nu)  = t-\maxNFs\sum_{(x,y)\in\Nct}s_{x,y}^{(t)}\eta_{x,y},
	\end{equation}
    and that a similar result applies to every $W_p^{(t)}$ distance.
\end{remark}

As a preliminary result, we show an \textit{a-posteriori} upper bound on the error we make by approximating $W_1$ with $W_{1}^{(t)}$.%, which is an interesting result in itself.

\begin{theorem}
\label{thm:exitimationabserr}
Let be given a threshold $t \ge 0$, $\mu,\nu \in \PP(X)$, and the complementary cost utility function $s^{(t)}_{x,y}=t-|x-y|$. 
\begin{itemize}
    \item[(i)] If there exists $\eta^* \in \MNFs$ such that $|\eta^*|_{N_t}=1$, then
        \begin{equation}
        \label{eq:w1equal}
            W_1^{(t)}(\mu,\nu)= W_1(\mu,\nu).
        \end{equation}
    \item[(ii)] For all $\eta\in \MNFs$ such that $|\eta|_{N_t}<1$
        \begin{align}
        \label{eq:fin:abserr}
            \nonumber W_1\mn&-W_1^{(t)}\mn \\
            % =\min_{\pi \in \Pimn}\mathbb{T}_c(\pi) - \min_{\pi \in \Pimn}\mathbb{T}_{\ct}(\pi) \\
            \nonumber &\leq \sum_{(x,y)\in X\times Y} (|x-y|-t)_+ \frac{ \Big(\mu_x 
            - \sumOct \eta_{x,y} \Big) \Big(\nu_y - \sumIct \eta_{x,y} \Big)} 
            {1-\sumXY\eta_{x,y}},
        \end{align}
        where $(\;\cdot\;)_+:=\max\{\cdot,0\}$ is the positive part of $x\in\erre$.
\end{itemize} 
\end{theorem}

\begin{proof}
If $\eta^*\in \MNFs$ satisfies $|\eta^*|_{N_t}=1$,
then $\eta^*$ is a transportation plan between $\mu$ and $\nu$. Since by definition $\eta^*=0$ if $|x-y|>t$, using \eqref{eq:thtrunceq} in the last equality, we obtain
\begin{align*}
    W_1(\mu,\nu) &\le \sumXY |x-y|\eta^*_{x,y}=\sum_{(x,y)\in N_t} |x-y|\eta^*_{x,y} \\
        & = \sum_{(x,y)\in N_t} (|x-y|-t +t)\eta^*_{x,y}\\
        &= t-\maxNFs\sum_{(x,y)\in\Nct}s_{x,y}^{(t)}\eta_{x,y}= W_1^{(t)}(\mu,\nu),
\end{align*}
which, together with \eqref{eq:w1trunc} concludes the proof of (i). 

Let now $\eta\in \MNFs$ satisfy $|\eta|_{N_t}<1$ and let $(\bar x,\bar y)\in\Nct$ satisfy $|\bar x-\bar y|<t$. % be such that $|\bar x-\bar y|<t$. 
% Then, 
We now show that if  $\sum_{x\in B_t(\bar y)}\eta_{x,\bar y} < \nu_{\bar y}$, then $\sum_{y\in B_t(\bar x)}\eta_{\bar x,y}= \mu_{\bar x}$ (and, in the same way, if $\sum_{y\in B_t(\bar x)}\eta_{\bar x,y}< \mu_{\bar x}$, then $\sum_{x\in B_t(\bar y)}\eta_{x,\bar y} = \nu_{\bar y}$).
Toward a contradiction, assume that
\begin{equation}
\label{cor:1}
% \text{either}
\quad \sum_{x\in B_t(\bar y)}\eta_{x,\bar y} < \nu_{\bar y}
    \quad \mbox{and} \quad 
\sum_{y\in B_t(\bar x)}\eta_{\bar x,y}< \mu_{\bar x}.
\end{equation}
% 
% We show \eqref{cor:1} by contradiction: let $\varepsilon>0$ be such that
% 
Then, there exists a $\varepsilon>0$ such that
\[
    \sum_{x\in B_t(\bar y)}\eta_{x,\bar y}+\varepsilon < \nu_{\bar y} \quad\quad \text{and} \quad\quad  \sum_{y\in B_t(\bar x)}\eta_{\bar x,y}+\varepsilon < \mu_{\bar x}
\]
and define $\eta^*$ as
\[
\eta^*_{x,y} := \left\{ 
\begin{array}{ll}
     \eta_{x,y}+\varepsilon & \quad  \mbox{ if } (x,y)=(\bar x,\bar y), \\
     \eta_{x,y} & \quad \; \mbox{otherwise}.
\end{array}
\right.
\]
It is easy to see that $\eta^*\in\NFs$, moreover, since $s^{(t)}_{\bar x,\bar y}=t-|\bar x-\bar y|>0$, we have that
\begin{eqnarray*}
    \sumNct s^{(t)}_{x,y}\eta^*_{x,y}=\sumNct s^{(t)}_{x,y}\eta_{x,y} +\varepsilon s^{(t)}_{\bar x,\bar y}>\sumNct s^{(t)}_{x,y}\eta_{x,y},
\end{eqnarray*}
which contradicts the maximality of $\eta$. We conclude that \eqref{cor:1} holds.
Let $\zeta\in\MM(X\times X)$ be defined as
\[
\zeta_{x,y}=\frac{ \Big(\mu_x 
- \sum_{v \in B_t(x)} \eta_{x,v} \Big) \Big(\nu_y - \sum_{u\in B_t(y)} \eta_{u,y} \Big)} 
{1-|\eta|_{N_t}}.
\]
Owing to \eqref{cor:1}, we have that $\zeta_{x,y}=0$ if $|x-y|<t$, while $\eta_{x,y}=0$ if $|x-y|>t$. When $|x-y|=t$ both measures can be non zero.
Define
\begin{equation}
\label{reconstructionoftheplane}
    \pi\in\MM(X\times X),\qquad \pi_{x,y} := \eta_{x,y} + \zeta_{x,y}.
\end{equation}
For $y\in X$ we compute
\begin{eqnarray*}
    \sum_{x\in X}\pi_{x,y}
%    &=& \sumIct \eta_{x,y} + \sum_{x\in X\backslash B_t(y)} \zeta_{x,y} = \sumIct \eta_{x,y} + \sum_{x \in X}\zeta_{x,y}\\
    &=&  \sumIct \eta_{x,y} + \sum_{x \in X}\zeta_{x,y}\\
    &=&  \sumIct \eta_{x,y} + \frac{  \left(\nu_y - \sum_{u\in B_t(y)} \eta_{u,y} \right)} 
{1-|\eta|_{N_t}}\sum_{x \in X}\left(\mu_x 
- \sum_{v \in B_t(x)} \eta_{x,v} \right)\\
    &=& \sumIct \eta_{x,y} + \nu_y - \sum_{u\in B_t(y)} \eta_{u,y}=\nu_y.
\end{eqnarray*}
Similarly, for all $x\in X$, $\sum_{y\in X} \pi_{x,y} = \mu_x$, thus $\pi\in\Pimn$.
Lastly, let $\pi^*_{x,y}$ be an optimal transportation plan between $\mu$ and $\nu$ with respect to the Euclidean metric and let $\pi$ be the transportation plan defined in \eqref{reconstructionoftheplane}. We have
 \begin{align*}
 	% \mathbb{T}_c(\pi^*)&=&
    \sumXY &|x-y|\pi^*_{x,y}\leq \sumXY |x-y|\pi_{x,y}\\
    &= \sumXY\bigg(\max\{|x-y|-t, 0\} + \min \{|x-y|, t\} \bigg) \pi_{x,y}\\
 	&= 	\sumXY (|x-y|-t)_+ \zeta_{x,y} + 
 	\sumXY \min \{|x-y|, t\}\pi_{x,y}\\
    &=\sumXY (|x-y|-t)_+\zeta_{x,y} + 
	t \sumXY \zeta_{x,y} + \sumNct |x-y| \eta_{x,y}\\
    &=\sumXY(|x-y|-t)_+\zeta_{x,y} + 
  	t \bigg( 1 - \sumNct \eta_{x,y} \bigg)+ \sumNct |x-y| \eta_{x,y}\\
    &=\sumXY(|x-y|-t)_+\zeta_{x,y} +  
    t -\sumNct s_{x,y}^{(t)} \eta_{x,y},
\end{align*}
so that
\[
  W_1(\mu,\nu)\leq \sumXY(|x-y|-t)_+\zeta_{x,y} +  
    t -\sumNct s^{(t)}_{x,y} \eta_{x,y}.
\]
Since $\eta \in \MNFs$, from Remark \ref{rmk:equivalence}, we infer
\[
	W_1(\mu,\nu)\leq\sumXY(|x-y|-t)_+\zeta_{x,y}+W_1^{(t)}(\mu,\nu),
\]
which allows us to conclude the proof.
% 
% which leads to 
% 
% \begin{equation}
% \label{disuguaglianzaimportante}
% 	 \min_{\pi \in \Pimn}\mathbb{T}_c(\pi)-\min_{\pi \in \Pimn}\mathbb{T}_{\ct}(\pi)\leq \min_{\zeta \in \FFseta}\mathbb{S}^{[t]}_c(\zeta).
% \end{equation}
% 
% In particular, since 
% \begin{equation}
% \label{eq:ff_indip}
%   \tilde{\zeta}_{x,y}:=  \frac{ \Big(\mu_x 
% - \sumOct \eta^*_{x,y} \Big) \Big(\nu_y - \sumIct \eta^*_{x,y} \Big)} 
% {1-\sum_{(x,y)\in X\times Y}\eta^*_{x,y}}
% \end{equation}
% is a feasible faraway flow, we get the following estimation.
\end{proof}

\begin{remark}
The upper bound proposed in Theorem \ref{thm:exitimationabserr} can be generalized to an upper bound on the relative error.
Indeed, since $W_1^{(t)}(\mu,\nu)\le W_1(\mu,\nu)$, given any $\mu,\nu\in \PP(X)$, $\eta\in\MNFs$, and $t \ge 0$, we have
\begin{equation}
\label{eq:fin:error_rel_tron}
    \dfrac{|W_1\mn-W_1^{(t)}\mn|}{|W_1\mn|}\leq\dfrac{\sumXY(|x-y|-t)_+\tilde{\mu}_x\tilde{\nu}_y}{\Big(1-\sumNct \eta_{x,y}\Big)W_1^{(t)}\mn}
\end{equation}
where $\tilde{\mu}_x=\mu_x-\sumOct\eta_{x,y}$ and $\tilde{\nu}_y=\nu_y-\sumIct\eta_{x,y}$.
\end{remark}

%As direct consequence to Theorem \ref{thm:exitimationabserr}, we infer that if there exists an optimal nearby flow whose total mass is equal to $1$ then the truncated Wasserstein Distance is equal to the Wasserstein Distance.

%\begin{corollary}\label{fullflowlemma}
%Let be given a positive threshold ${\color{blue} t \ge 0}$ and two probability measures $\mu$ and $\nu$. 
% 
%If the nearby flow $\eta \in \Theta^{(t)}_{\mu,\nu}$ such that $|\eta|_{\Nct}=1$
%% \begin{equation}
%% \label{eq:etasum1}
%%     \sumNct \eta^*_{x,y} = 1,
%% \end{equation}
%then 
%\[
%W_1(\mu,\nu)=W_1^{(t)}(\mu,\nu).
%\]
%\end{corollary}

%\begin{proof}
%It follows from Theorem \ref{thm:exitimationabserr}. 
% 
%Indeed, let $C_t:=\max_{(x,y)\in\Nct}(|x-y|-t)_+$, then we have
%\begin{align*}
%W_1(\mu,&\nu)-W_1^{(t)}(\mu,\nu)\\
%    &\le\sumXY (|x-y|-t)_+ \frac{ \Big(\mu_x 
%- \sumOct \eta^*_{x,y} \Big) \Big(\nu_y - \sumIct \eta^*_{x,y} \Big)} 
%{1-\sumNct\eta^*_{x,y}}\\
%&\le C_t \sumXY \frac{ \Big(\mu_x 
%- \sumOct \eta^*_{x,y} \Big) \Big(\nu_y - \sumIct \eta^*_{x,y} \Big)} 
%{1-\sumNct\eta^*_{x,y}}\\
%&\le C_t \frac{ \sumX\Big(\mu_x 
%- \sumOct \eta^*_{x,y} \Big) \sumY\Big(\nu_y - \sumIct \eta^*_{x,y} \Big)} 
%{1-\sumNct\eta^*_{x,y}}\\
%&=C_t\Big(1-\sumNct\eta^*_{x,y}\Big).
%\end{align*}
%Thus, if $|\eta|_{\Nct}=1$, we infer $W_1(\mu,\nu)=W_1^{(t)}(\mu,\nu)$.
%\end{proof}

% 
We are now ready to state the main result of this section, which ties the \minimal saturation threshold, defined in \eqref{minimalsaturationformula}, to the Wasserstein Projection problem defined in Section \ref{sec:prelim_projection}.

\begin{theorem}
\label{thm:projectionclassic}
    Let $t \ge 0$ be a threshold, $\mu\in \PP(X)$, and $f\in\MM(X)$ with $|f|_X\geq 1$. 
    Given the complementary cost utility function $s^{(t)}_{x,y}=t-|x-y|$, we consider the following problem:
    \begin{align}
    \label{eq:projthm}
        \max_{\eta\in\MM(X\times X)}&\sumNct s^{(t)}_{x,y}\eta_{x,y}\\
        \nonumber    \text{subj.} \quad &\sumOct \eta_{x,y}\le \mu_x \quad \text{and} \quad \sumIct \eta_{x,y}\le f_y \quad \forall x,y\in X.
    \end{align}
    If $\eta\in\MM(X\times X)$ is a solution to \eqref{eq:projthm} such that $|\eta|_{X\times X}=1$, then the second marginal of $\eta$
    \begin{equation}
        \label{eq:zeta_projdef}
        \zeta_y=\sumIct\eta_{x,y},\quad\quad\quad \forall y \in X
    \end{equation}
    is a projection of $\mu$ on $K_f$, with respect to $W_1$. 
    That is, 
    \[
    	\zeta\in K_f=\left\{\rho\in\PP(X) : \rho_x\leq f_x,\quad
    \forall x \in X\right\}
    \] 
    and
    \begin{equation}
    \label{eq:zeta_proj}
        W_1(\mu,\zeta) \leq W_1(\mu,\rho) \quad\quad \quad \forall \rho \in K_f.
    \end{equation}
\end{theorem}

\begin{proof}
    Let $\eta$ be a solution to \eqref{eq:projthm} such that $|\eta|_{X\times X}=1$ 
    then $\eta\in\Pi_{\mu,\zeta}$, where $\zeta$ is defined in \eqref{eq:zeta_projdef}.
    Note that $\zeta\in\PP(X)$, since $\zeta_y\geq 0$ and 
    \[
 	   \sum_{y\in X}\zeta_{y}=\sum_{y\in X}\sum_{x\in X}\eta_{x,y}=\sum_{x\in X}\mu_x=1.
    \]
    Moreover,  $\zeta\in K_f$ by definition of $\zeta$ and the fact that $\eta$ satisfies the constraints in \eqref{eq:projthm}.
    % 
    % Owing to Remark \ref{rmk:equivalence}
    % 
   Owing to Theorem \ref{thm:exitimationabserr}, we obtain
    % $\eta$ is optimal between $\mu$ and $\zeta$, hence, from Corollary \ref{fullflowlemma}, we have that
    \begin{equation}
    \label{eq::muzeta}
        W_1^{(t)}(\mu,\zeta)=W_1(\mu,\zeta).
    \end{equation}
    Lastly, we show that \eqref{eq:zeta_proj} holds.
    Toward a contradiction, let us assume that there exists $\rho\in K_f$ such that $W_1(\mu,\rho)<W_1(\mu,\zeta)$, then 
    % 
    % $\zeta \notin P_{K_f}(\mu)$. For any $\rho \in P_{K_f}(\mu)$ we have then 
    \[
        W_1^{(t)}(\mu,\rho)\leq W_1(\mu,\rho) < W_1(\mu,\zeta),
    \]    
    where the first inequality comes from the inequality $|x-y|_t\le |x-y|$.
    Moreover, owing to \eqref{eq::muzeta}
    \[
        W_1^{(t)}(\mu,\rho) \le W_1(\mu,\rho) < W_1(\mu,\zeta) = W_1^{(t)}(\mu,\zeta),
    \]    
    thus, from Remark \ref{rmk:equivalence}, we find a nearby flow $\eta^*\in\mathcal{N}^{(t)}_{\mu,\rho}$ such that
    \[
    W_1^{(t)}(\mu,\rho)=t-\sumNct s^{(t)}_{x,y}\eta^*_{x,y}<t-\sumNct s^{(t)}_{x,y}\eta_{x,y}
    \]
	and therefore
    \[
    		\sumNct s^{(t)}_{x,y}\eta_{x,y}<\sumNct s^{(t)}_{x,y}\eta^*_{x,y}.
    \]
    Since $\rho\in K_f$, we conclude that $\eta^*$ is a feasible solution for problem \eqref{eq:projthm}, which contradicts the optimality of $\eta$.
\end{proof}

\begin{remark}
    % By replacing $|x-y|$ with $|x-y|^p$ in the utility function $s^{(t)}_{x,y}$ and $|x-y|_t$ with $|x-y|_t^p$, Theorem \ref{thm:projectionclassic} can be generalised to compute the projection of a probability measure on $K_f$ with respect to any $p$-Wasserstein distance $W_p$, where $p\in[1,\infty)$.
    % 
    Notice that all the arguments used so far can be adapted to the case in which $|x-y|$ is replaced with any positive and bounded cost function $c_{x,y}$ in the utility function $s^{(t)}_{x,y}$ and $|x-y|_t$ with $(c_{x,y})_t:=\min\{c_{x,y},t\}$.
    In this case, we are able to compute the $W_{c,\infty}$ distance, defined as
    \[
        W_{c,\infty}(\mu,\nu)=\min_{\gamma\in\Pi(\mu,\nu)}\max_{\gamma_{x,y}>0}c_{x,y}.
    \]
    Similarly, we can compute the solution to the minimization problem
    \[
        \argmin_{\nu\in K_f}W_c(\mu,\nu)=\argmin_{\nu\in K_f}\min_{\gamma\in\Pi(\mu,\nu)}c_{x,y}\gamma_{x,y}.
    \]
    % Notice however, that if $c_{x,y}$ is generic, there is
    % 
    
    % 
    In particular, Theorem \ref{thm:projectionclassic} can be generalised to compute the infinity Wasserstein distance induced by any $l_p$ norm and the a projection of a probability measure on $K_f$ with respect to any $p$-Wasserstein distance $W_p$, where $p\in[1,\infty)$.
\end{remark}

To conclude the section, we adapt Theorem \ref{thm:projectionclassic} to the case $p=\infty$.
%in which we want to compute the projection of $\mu$ onto $K_f$ according to $W_\infty$.
% 

\begin{corollary}
\label{crll:lastcorollary}
    Let  $\mu\in \PP(X)$ be a probability measure and $f\in\MM(X)$ be a measure such that $|f|_X\geq 1$.
    Let $\tau_s(\mu,f)\in\erre$ be the \minimal saturation threshold between $\mu$ and $f$ with respect to the constant utility function $s^{(t)}_{x,y}=1$.
    % $s^{(t)}_{x,y}=1$ for every $t\ge 0$ and every $(x,y)\in\Nct$.
    % 
    Then,
    \begin{equation}
        \label{eq:winfprojcor}
        \tau_s(\mu,f)=\min_{\nu\in K_f}W_{\infty}(\mu,\nu).
    \end{equation}
    Moreover, let $\eta$ be a solution to 
    \begin{align}
    \label{eq:projthminfinity}
        \max_{\eta\in\MM(X\times X)}&\sum_{(x,y)\in N_T} \eta_{x,y}\\
        \nonumber\text{subj.} \quad&\sum_{y\in B_T(x)} \eta_{x,y}\le \mu_x \quad \text{and} \quad \sum_{x\in B_T(y)} \eta_{x,y}\le f_y \quad \forall x,y\in X,
    \end{align}
    where $T=\tau_s(\mu,f)$.
    Then $\zeta_y=\sum_{x\in B_T(y)}\eta_{x,y}$ is a projection of $\mu$ on $K_f$, with respect to $W_{\infty}$. 
\end{corollary}

\begin{proof}
First, we show that $\min_{\nu\in K_f}W_\infty(\mu,\nu)\le \tau_s(\mu,f)$.
Let $t$ be such that there exists $\eta\in\MM(X\times X)$ with $|\eta|_{\Nct}=1$,
\[
    \sumOct \eta_{x,y}\le \mu_x,\quad\quad\text{and}\quad\quad \sumIct \eta_{x,y}\le f_y.
\]
Then $\eta$ is a transportation plan between $\mu$ and $\zeta$, where 
\[
    \zeta_y=\sumOct\eta_{x,y}.
\]
Then $W_{\infty}(\mu,\zeta)\le t$ and therefore $W_\infty(\mu,\zeta)\le\tau_s(\mu,f)$.
We now prove the other inequality.
Let $\zeta\in K_f$ be a projection of $\mu$ on $K_f$ with respect to $W_\infty$.
Let $T:=W_\infty(\mu,\zeta)$ and, towards a contradiction, let us assume that $T<\tau_s(\mu,f)$. Then there exists $\pi\in\Pi_{\mu,\zeta}$ such that $\pi_{x,y}=0$ if $|x-y|>T$. 
Let $\eta^*_{x,y}=\pi_{x,y}$ for every $(x,y)\in N_T$,  
then $\eta^*$ is feasible and optimal for problem \eqref{eq:projthminfinity}.
We then conclude that $\tau_s(\mu,f)\le T$ and thus the proof.
\end{proof}

\begin{remark}
    It is important to note that the projection problem on the right-hand side of \eqref{eq:winfprojcor} generally does not have a unique solution. 
    Therefore, Corollary \ref{crll:lastcorollary} identifies one possible projection among many.
\end{remark}

\section{Experimental Results}

In this section, we introduce the Climbing Algorithm, a routine capable of retrieving the \Minimal Saturation Threshold.
We show that it can be used to compute the infinity Wasserstein distance and to solve the Wasserstein projection problem.
In the numerical experiments, all optimisation models were solved using IBM CPLEX Optimisation Studio (version 22.1.1) through its Python API. The test instances were derived from greyscale images (represented as two-dimensional histograms) provided by the DOTMark benchmark~\cite{Schrieber2016DOTmarkA}.

\subsection{The Climbing Algorithm}
First, we define the routine of the Climbing Algorithm.
The Climbing Algorithm takes as input two probability measures $\mu,\nu\in\PP(X)$, a utility function $s^{(t)}$ over $N_t$, and a set of positive thresholds $\mathcal{T}:=\{t_i\}_{i=1,\dots,n}$, with $t_i < t_{i+1}$, and returns the smallest element of $\mathcal{T}$ for which there exists a solution to the Nearby Flow Problem (defined as in \eqref{nbf:0}) i.e., the smallest $\bar t \in \mathcal{T}$ for which there exists $\eta\in\MNFs$ such that $|\eta|_{N_{\bar t}}=1$.
To accomplish this, we iteratively solve the Maximum Nearby Flow for growing values of $t\in\mathcal{T}$, and we stop when we find a solution with total mass equal to $1$.
In Algorithm \ref{alg:pure_climb}, we report the pseudo-code of the Climbing Algorithm.
% 

% 
% Given a utility function $s^{(t)}_{x,y}$, the idea of the algorithm is to solve the MNF for a sequence of increasing thresholds until the total mass of the solution found is equal to $1$ or close to $1$ up to an accuracy parameter $\delta>0$.
% % 
% We report in Algorithm \ref{alg:pure_climb} the pseudo-code of the Climbing Algorithm.
% 

% 
Notice that if $\mathcal{T}$ is the image of $c_{x,y}$ for every $(x,y)\in X\times X$, the Climbing Algorithm always returns the \minimal saturation threshold.

In our experiments, we start from $t_0=\min \mathcal{T}$ and, at each iteration, we set $t_{k+1}=\min\{c_{x,y}\; \text{s.t.}\; c_{x,y}>t_k\}$.
While there are alternative approaches to search for the optimal value $t\in\mathcal{T}$, we opted for this simple search routine as it well-suited to search for minimal saturation thresholds that are expected to be small.
In general, depending on the problem at hand, it would be possible to improve the performances of the Climbing Algorithm by tuning the searching method (\textit{e.g.} it could leverage a multiscale approach or a bisection approach).

\begin{algorithm}[t!]
\caption{The Climbing Algorithm }\label{alg:pure_climb}
\begin{algorithmic}
\Require two measures $\mu,\nu$, a cost function $c$, a utility function $s$, a tolerance threshold $\delta$, and the image set of the cost function ordered increasingly $\mathcal{T}:=\{t_i\}_{i=1,\dots,n}$
\Ensure the \minimal saturation threshold  $\tau_s\mn$associated to the utility $s$ and the two measures $\mu$ and $\nu$ 
% $W_c\mn$, the optimal transportation plan $\pi$, and an upper bound on the $L_\infty$ norm of $c$ with respect to the measure $\pi$, up to an error proportional to $\epsilon$.
\State $\Delta m\gets 1$
% \State $\tau \gets \min \mathcal{T}$
\While{$\Delta m \geq \delta$}
    \State $\tau \gets \min \mathcal{T}$
    \State $\mathcal{T}\gets \mathcal{T}\backslash \{\tau\}$
    \State $\eta\gets \mbox{Maximize functional $\mathbb{B}_s^{(\tau)}(\eta)$}$
    \State $\Delta m\gets 1-\sum \eta$
\EndWhile
\State Return $\tau$
\end{algorithmic}
\end{algorithm}

\subsection{Computing the infinity Wasserstein Distance}
We calculate the $W_{\infty}$-distance between each pair of images within a category of the DOTmark dataset. 
Specifically, for each pair of images with resolution $N$, we normalize them to define two measures, namely $\mu = \{\mu_{i,j}\}_{i,j=1,\dots, N}$ and $\nu = \{\nu_{i,j}\}_{i,j=1,\dots, N}$.
We then apply the Climbing Algorithm with $s^{(t)} \equiv 1$.
To keep $W_\infty$ bounded, independently of the resolution, we measure the distance between two pixels as follows
\begin{equation}
\label{eq:cost_exp}
    c((i,j),(k,l))=\frac{1}{N}\sqrt{(i-k)^2+(j-l)^2},
\end{equation}
where $N$ is the resolution of the considered image.
According to Theorem~\ref{thm:alternativeWinfformulation}, the $W_{\infty}$-distance corresponds to the output of the Climbing Algorithm. 
In Table~\ref{tab:winfty}, the mean and maximum values of the $W_{\infty}$-distance are presented for various categories.
For completeness, in Table~\ref{tab:winftytimes} we report the average and maximum times needed to compute the $W_\infty$-distance for each couple of images in the categories considered in Table~\ref{tab:winfty}.
From our results, we observe that the ``Shapes'' class has an average $W_\infty$-distance larger than the other classes, leading to a larger average computational time.
We believe this is because the ``Shapes'' class contains elements that are black-and-white images without intermediate greyscale pixels.
In accordance with this conjecture, we observe that for the ``White Noise'' class the average $W_\infty$-distance and the computational time are up to two orders of magnitude lower than the other counterparts, for the $32 \times 32$ resolution, as well as for the $64 \times 64$ one.

\begin{table}[htbp]
\centering
\small
\begin{tabular}{lcccccc}
\toprule
Category & $32 \times 32$ & $64 \times 64$ & %$128 \times 128$
\\
\midrule
\begin{filecontents*}{tabDWinf.csv}
Category,mean32,max32,mean64,max64
ClassicImages,0.1378291579695463,0.2275034340400162,0.12911277656506553,0.22371595411369302
Shapes,0.26485804176412475,0.4341388746703064,0.2660314715869744,0.4352621336620037
WhiteNoise,0.06062995540928144,0.06987712429686843,0.03235732878463336,0.04419417382415922
\end{filecontents*}
\csvreader[
late after line=\\,
]
{tabDWinf.csv}{1=\cat,2=\avgi, 3=\maxi, 4=\avgii, 5=\maxii, 6=\avgiii, 7=\maxiii}
{\cat & \numD{\avgi} (\numD{\maxi}) & \numD{\avgii} (\numD{\maxii}) & %\num{\avgiii} (\num{\maxiii})
}
\bottomrule
\end{tabular}
\caption{Mean values of the $W_{\infty}$-distance, computed over all images within a single category. The maximum values are indicated in parentheses.}
\label{tab:winfty}
\end{table}
\begin{table}[htbp]
\centering
\small
\begin{tabular}{lcccccc}
\toprule
Category & $32 \times 32$ & $64 \times 64$ & %$128 \times 128$ 
\\
\midrule
\begin{filecontents*}{tabTWinf.csv}
Category,mean32,max32,mean64,max64
Classic Images,33.29899399545457,186.60020661354065,3646.382469532225,21770.772467136383
Shapes,66.39902115927802,714.8660731315613,9763.92957707511,67171.4168138504
White Noise,1.3632234891255697,2.0613038539886475,43.32371188269721,87.96237421035767
\end{filecontents*}
\csvreader[
late after line=\\,
]
{tabTWinf.csv}{1=\cat, 2=\avgi, 3=\maxi, 4=\avgii, 5=\maxii, 6=\avgiii, 7=\maxiii}
{\cat & \numD{\avgi} (\numD{\maxi}) & \numD{\avgii} (\numD{\maxii}) & %\num{\avgiii} (\num{\maxiii})
}
\bottomrule
\end{tabular}
\caption{Mean computation times (in seconds) for the $W_{\infty}$-distance, computed over all images within a single category. The maximum values are indicated in parentheses.}
\label{tab:winftytimes}
\end{table}

\subsection{Solving the Wasserstein Projection Problem}
To evaluate the effectiveness of projection computations with respect to the infinity Wasserstein distance, we run the following experiment.
Given a resolution $N=32,64,$ and $128$, for each $N \times N$ image in the ``Classic Images'' and ``Shapes'' categories of the DOTmark dataset, we normalize the image in order to define a probability measure $\mu = \{ \mu_{i,j} \}_{i,j=1,\dots,N}$.
% 
% \textcolor{red}{As per the previous experiment, we measure the distance between two pixels as in \eqref{eq:cost_exp}.}
% 
Subsequently, each measure is padded with a strip of $N/2$ zeros (white pixels) on all four sides, producing a $2N \times 2N$ grid.
A projection of this measure is then calculated on the set $K_{\theta,\mu}$, defined by
\begin{equation*}
    K_{\theta,\mu}: = \Big\{ \nu: \ \nu_{i,j} \leq \theta \max_{i,j=1,\ldots,2N} \mu_{i,j} \Big\},
\end{equation*}
where $0 < \theta < 1$ serves as a limiting parameter that ranges in the set
\linebreak $\{0.975,0.95,0.925,0.9,0.85,0.8\}$.
As per the previous experiment, we measure the distance between two pixels as in \eqref{eq:cost_exp}.
In Tables~\ref{tab:ClaProjWinftyThresholds} and~\ref{tab:ClaProjWinftyTime} we report our results.
In particular, Table~\ref{tab:ClaProjWinftyThresholds} shows the mean and maximum values of the \minimal saturation threshold retrieved by the Climbing Algorithm for the $W_{\infty}$-projection calculated for each image considered.
We recall that owing to Corollary \ref{crll:lastcorollary} the \minimal saturation threshold is equal to the $W_\infty$ distance between the initial image and its projection.
Moreover, the higher the \minimal saturation threshold, the higher the number of iterations needed to compute the projection.
Table~\ref{tab:ClaProjWinftyTime} shows the mean and maximum times required to compute the $W_{\infty}$-projection for different values of the limiting parameter $\theta$. 
% 

% \begin{table}[t]
% \centering
% \scriptsize
% \begin{tabular}{ccccccc}
% \toprule
% \multirow{2}{*}{$\theta$} & \multicolumn{3}{c}{Classic Images} & \multicolumn{3}{c}{Shapes} \\ 
% \cmidrule(lr){2-4} \cmidrule(lr){5-7}
% & $32 \times 32$ & $64 \times 64$ & $128 \times 128$ & $32 \times 32$ & $64 \times 64$ & $128 \times 128$\\
% \midrule
% \csvreader[
% late after line=\\,
% ]
% {tables/tabProjWinftyThresholds.csv}{1=\barr, 2=\avgi, 3=\maxi, 4=\avgii, 5=\maxii, 6=\avgiii, 7=\maxiii, 8=\avgiv, 9=\maxiv, 10=\avgv, 11=\maxv, 12=\avgvi, 13=\maxvi}
% {\numDi{\barr} & \numD{\avgi} (\numD{\maxi}) & \numD{\avgii} (\numD{\maxii}) & \numD{\avgiii} (\numD{\maxiii}) & \numD{\avgiv} (\numD{\maxiv}) & \numD{\avgv} (\numD{\maxv}) & \numD{\avgvi} (\numD{\maxvi})}
% \bottomrule
% \end{tabular}
% \caption{Mean threshold values $\bar{t}$ reached by the Climbing Algorithm in the computation of the $W_{\infty}$-projection. Maximum values are provided in parentheses.}
% \label{tab:ClaProjWinftyThresholds}
% \end{table}

\begin{table}[t]
    \centering
    \begin{tabular}{c c c c c}
        \toprule
        & $\theta$ & $32 \times 32$ & $64 \times 64$ & $128 \times 128$  \\
        \midrule
        \multirow{6}{*}{\rotatebox{90}{Classic Images}}     
         & \numDi{0.975} & \numD{3.1e-2} (\numD{3.1e-2}) & \numD{1.6e-2} (\numD{1.6e-2}) & \numD{7.8e-3} (\numD{7.8e-3}) \\ 
         & \numDi{0.95} & \numD{3.1e-2} (\numD{3.1e-2}) & \numD{1.6e-2} (\numD{1.6e-2}) & \numD{7.8e-3} (\numD{7.8e-3}) \\ 
         & \numDi{0.925} & \numD{3.1e-2} (\numD{3.1e-2}) & \numD{1.6e-2} (\numD{1.6e-2}) & \numD{7.8e-3} (\numD{7.8e-3}) \\ 
         & \numDi{0.9} & \numD{3.1e-2} (\numD{3.1e-2}) & \numD{1.6e-2} (\numD{1.6e-2}) & \numD{7.8e-3} (\numD{7.8e-3}) \\ 
         & \numDi{0.85} & \numD{3.1e-2} (\numD{3.1e-2}) & \numD{1.6e-2} (\numD{1.6e-2}) & \numD{8.1e-3} (\numD{1.1e-2}) \\ 
         & \numDi{0.8} & \numD{3.1e-2} (\numD{3.1e-2}) & \numD{1.6e-2} (\numD{1.6e-2}) & \numD{9.6e-3} (\numD{1.7 e-2}) \\ 
        \midrule
        \multirow{6}{*}{\rotatebox{90}{Shapes}}     
         & \numDi{0.975} & \numD{3.1e-2} (\numD{3.1e-2}) & \numD{1.6e-2} (\numD{1.6e-2}) & \numD{7.8e-3} (\numD{7.8e-3}) \\ 
         & \numDi{0.95} & \numD{3.1e-2} (\numD{3.1e-2}) & \numD{1.6e-2} (\numD{1.6e-2}) & \numD{1.1e-2} (\numD{1.6e-2}) \\ 
         & \numDi{0.925} & \numD{3.1e-2} (\numD{3.1e-2}) & \numD{1.8e-2} (\numD{3.1e-2}) & \numD{1.3e-2} (\numD{2.3e-2}) \\ 
         & \numDi{0.9} & \numD{3.3e-2} (\numD{4.4e-2}) & \numD{2.3e-2} (\numD{3.5e-2}) & \numD{1.8e-2} (\numD{3.1e-2}) \\ 
         & \numDi{0.85} & \numD{3.8e-2} (\numD{6.3e-2}) & \numD{2.8e-2} (\numD{4.7e-2}) & \numD{2.6e-2} (\numD{4.7e-2}) \\ 
         & \numDi{0.8} & \numD{4.8e-2} (\numD{7.0e-2}) & \numD{3.9e-2} (\numD{6.6e-2}) & \numD{3.6e-2} (\numD{6.6e-2}) \\
         \bottomrule
    \end{tabular}
    \caption{Mean threshold values $\bar{t}$ reached by the Climbing Algorithm in the computation of the $W_{\infty}$-projection. Maximum values are provided in parentheses.}
\label{tab:ClaProjWinftyThresholds}
\end{table}
\begin{table}[htbp]
\centering
\scriptsize
\begin{tabular}{ccccccc}
\toprule
\multirow{2}{*}{$\theta$} & \multicolumn{3}{c}{Classic Images} & \multicolumn{3}{c}{Shapes} \\ 
\cmidrule(lr){2-4} \cmidrule(lr){5-7}
& $32 \times 32$ & $64 \times 64$ & $128 \times 128$ & $32 \times 32$ & $64 \times 64$ & $128 \times 128$\\
\midrule
\begin{filecontents*}{tabProjWinftyTime.csv}
tau,Cla 32 mean,Cla 32 max,Cla 64 mean,Cla 64 max,Cla 128 mean,Cla 128 max,Sha 32 mean,Sha 32 max,Sha 64 mean,Sha 64 max,Sha 128 mean,Sha 128 max
0.975,1.92005615234375,2.6331069469451904,18.164919233322145,18.325990200042725,376.30716495513917,382.155588388443,1.7878223180770874,1.884087085723877,18.187385988235473,18.40326833724976,378.0103766441345,381.3618052005768
0.95,1.8265625953674316,1.8698580265045168,18.19126660823822,18.39418244361877,376.5240755081177,379.85221123695374,1.7925809144973754,1.8648319244384768,18.212650847434997,18.494007110595703,719.2881477355957,1151.6944785118103
0.925,1.830225419998169,1.901881217956543,18.21907296180725,18.49963450431824,376.5596438407898,379.9106597900391,1.79164879322052,1.8627300262451167,25.45887157917023,53.55685591697693,1036.5600950717926,2297.9022958278656
0.9,1.8348968267440795,1.8836419582366943,18.200200057029726,18.4396493434906,376.448144698143,379.269517660141,1.9693640232086183,3.435898303985596,36.47902400493622,71.90904927253723,1711.6410085201264,3534.068994283676
0.85,1.8288179874420165,1.8969395160675049,18.340111470222475,18.91042447090149,423.8692853212357,849.6501502990723,2.712757444381714,5.287059307098389,53.309895515441895,108.69361162185668,3181.7655571699142,7485.867031574249
0.8,1.8381685256958007,1.8893017768859863,21.245966482162476,47.69848608970642,600.2655220270157,1811.5587766170504,3.8055861711502077,7.088803052902222,92.78758678436279,204.8574223518372,5641.768555760384,14321.536999940872
\end{filecontents*}
\csvreader[
late after line=\\,
]
{tabProjWinftyTime.csv}{1=\barr, 2=\avgi, 3=\maxi, 4=\avgii, 5=\maxii, 6=\avgiii, 7=\maxiii, 8=\avgiv, 9=\maxiv, 10=\avgv, 11=\maxv, 12=\avgvi, 13=\maxvi}
{\numDi{\barr} & \numD{\avgi} (\numD{\maxi}) & \numD{\avgii} (\numD{\maxii}) & \numD{\avgiii} (\numD{\maxiii}) & \numD{\avgiv} (\numD{\maxiv}) & \numD{\avgv} (\numD{\maxv}) & \numD{\avgvi} (\numD{\maxvi})}
\bottomrule
\end{tabular}
\caption{Mean computation times (in seconds) for the $W_{\infty}$-projection, with maximum values shown in parentheses.}
\label{tab:ClaProjWinftyTime}
\end{table}

From our results, we observe that in most cases, the threshold computed by the algorithm is $\frac{1}{N}$, where $N$ is the resolution of the image, which indicates that the climbing algorithm needs to solve only one Linear Programming (LP) problem.
% 
% when $N=32$ the projection is found when the threshold is equal to $1/32 \approx \numD{0.03125}$, thus the climbing algorithm needs to solve just one Linear Programming (LP) problem.}
% 
However, as the parameter $\theta$ decreases, the threshold increases.
This effect is more noticeable when we consider the images belonging to the ``Shapes'' category, as they have all their mass concentrated in a small portion of the image.
In contrast, the images belonging to the ``Classic Images'' class have their mass distributed more evenly, hence the \minimal saturation threshold we need to compute a projection is smaller.
As a result, the times needed to retrieve the projections change depending on the class we consider.
Indeed, retrieving the projections for the Shape class is more expensive, as it requires going through more thresholds and thus solving more Linear Programming problems.
For reference, we report a subset of images from the ``Shapes'' category, along with their corresponding $W_{\infty}$-projections for various values of $\theta$, in Figure~\ref{fig:ProjWinftyGrid}.
\begin{figure}[t!]
\centering
\includegraphics[width=\linewidth]{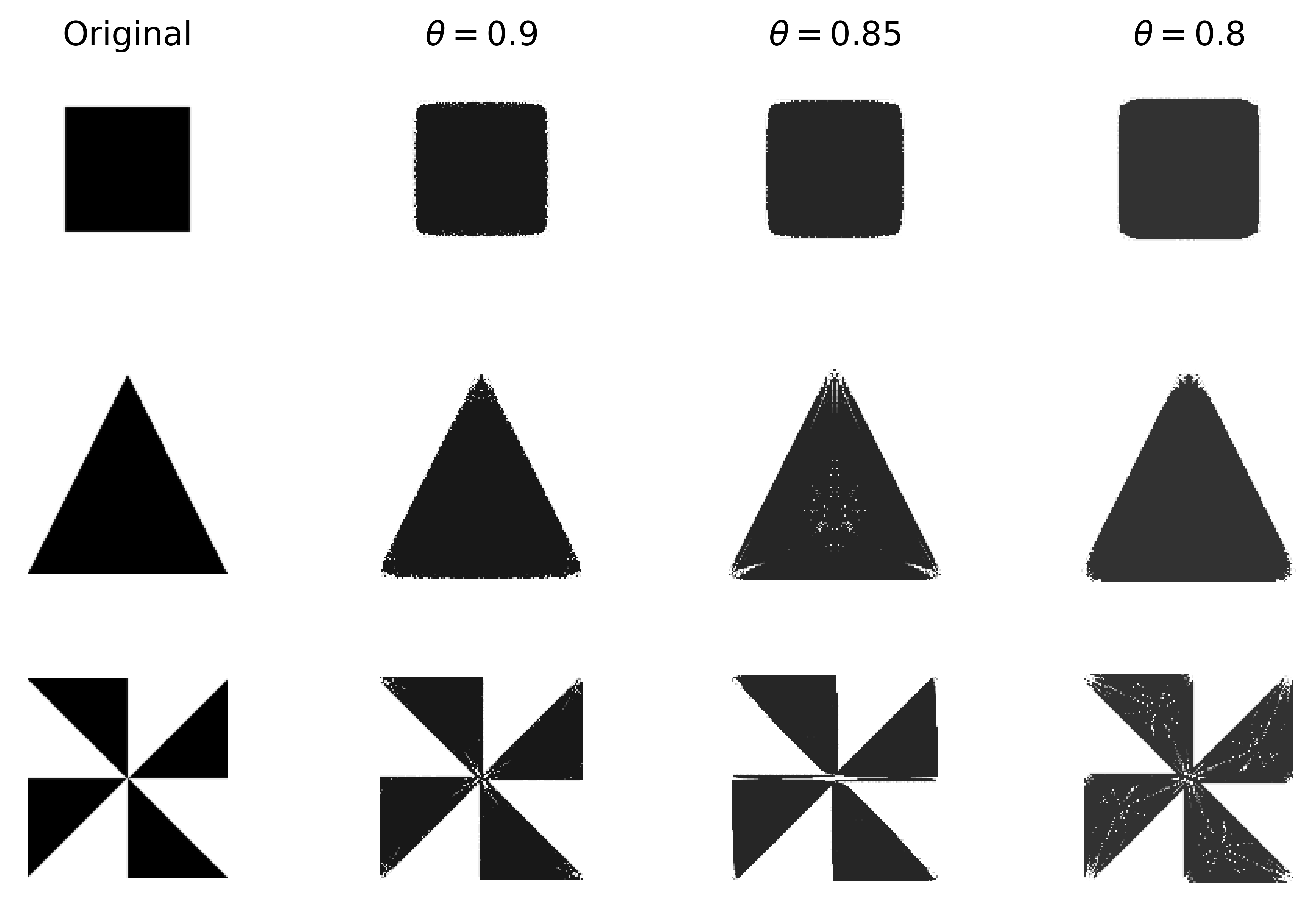}
\caption{$W_{\infty}$-projections of three  ``Shapes'' images for $\theta=0.9,0.85,0.8$.}
\label{fig:ProjWinftyGrid}
\end{figure}

\section{Conclusion}

In this work, we addressed the computational problems associated with the infinity Wasserstein distance and the projection of probability measures on closed subsets of probability spaces.
These problems are crucial in various applied fields, but often remain impractical due to their computational complexity or the absence of algorithms. 
To overcome these limitations, we introduced a novel class of Linear Programming formulations and a computational routine capable of efficiently determining both the infinity Wasserstein distance and the projections with respect to any $p$-Wasserstein distance for $p \in [1,\infty]$.
To validate our approach, we performed extensive numerical experiments on the DOTmark dataset, demonstrating the effectiveness of our algorithms.
% 

% \section{Conclusion}

% In this paper, we introduce a Linear Programming formulation aimed at addressing the discrete Optimal Transport problem induced by a truncated ground distance.
% % 
% The approach hinges on identifying nodes whose distances are less than a predetermined threshold $t$, resulting in a problem that is simpler to solve: the Maximal Nearby Flow problem.
% % 
% The main contribution of our study lies in deriving an \textit{a posteriori} estimation that quantifies the difference between the Wasserstein distance and its truncated counterpart. 
% % 
% This estimation is obtained by leveraging the complementary transport problem concerning points whose distance is greater or equal to $t$.
% % 
% We then define the Climbing Algorithm, which iteratively increases the threshold $t$ and solves the related Maximum Nearby Flow problem.
% % 
% By combining the routine of the Climbing Algorithm with the derived \textit{a posteriori} estimation, we are capable of computing the exact Wasserstein distance up to a desired level of precision.
% % 
% Notably, our algorithm encompasses also the Infinity Wasserstein distance.
% % 
% Lastly, we related the solution found by the Climbing Algorithm to the projection problem induced by the Wasserstein metrics.

%% If you have bib database file and want bibtex to generate the
%% bibitems, please use
%%
\bibliographystyle{elsarticle-num} 
\bibliography{biblio}

\end{document}